\documentclass[reqno]{gokova}
\usepackage{epsfig,graphicx}

\begin{document}
\def\E{\ifmmode{\mathbb E}\else{$\mathbb E$}\fi} %natural numbers
\def\N{\ifmmode{\mathbb N}\else{$\mathbb N$}\fi} %natural numbers
\def\R{\ifmmode{\mathbb R}\else{$\mathbb R$}\fi} %real numbers
\def\Q{\ifmmode{\mathbb Q}\else{$\mathbb Q$}\fi} %rational numbers
\def\C{\ifmmode{\mathbb C}\else{$\mathbb C$}\fi} %complex numbers
\def\H{\ifmmode{\mathbb H}\else{$\mathbb H$}\fi} %complex numbers
\def\Z{\ifmmode{\mathbb Z}\else{$\mathbb Z$}\fi} %integers
\def\P{\ifmmode{\mathbb P}\else{$\mathbb P$}\fi} %real numbers
\def\T{\ifmmode{\mathbb T}\else{$\mathbb T$}\fi} %real numbers
\def\SS{\ifmmode{\mathbb S}\else{$\mathbb S$}\fi} %real numbers
\def\DD{\ifmmode{\mathbb D}\else{$\mathbb D$}\fi} %real numbers

\renewcommand{\a}{\alpha}
\renewcommand{\b}{\beta}
\renewcommand{\d}{\delta}
\newcommand{\D}{\Delta}
\newcommand{\e}{\varepsilon}
\newcommand{\g}{\gamma}
\newcommand{\G}{\Gamma}
\newcommand{\la}{\lambda}
\newcommand{\La}{\Lambda}
\newcommand{\n}{\nabla}
\newcommand{\var}{\varphi}
\newcommand{\s}{\sigma}
\newcommand{\Sig}{\Sigma}
\renewcommand{\t}{\tau}
\renewcommand{\th}{\theta}
\renewcommand{\O}{\Omega}
\renewcommand{\o}{\omega}
\newcommand{\z}{\zeta}

\newcommand{\ben}{\begin{enumerate}}
\newcommand{\een}{\end{enumerate}}
\newcommand{\be}{\begin{equation}}
\newcommand{\ee}{\end{equation}}
\newcommand{\bea}{\begin{eqnarray}}
\newcommand{\eea}{\end{eqnarray}}
\newcommand{\bc}{\begin{center}}
\newcommand{\ec}{\end{center}}

\newtheorem{thm}{Theorem}[section]
\newtheorem{cor}[thm]{Corollary}
\newtheorem{lem}[thm]{Lemma}
\newtheorem{prop}[thm]{Proposition}
\newtheorem{ax}{Axiom}
\newtheorem{conj}[thm]{Conjecture}

\theoremstyle{definition}
\newtheorem{defn}{Definition}[section]

\theoremstyle{remark}
\newtheorem{rem}{\rm\bfseries{Remark}}[section]
\newtheorem*{notation}{Notation}

\newtheorem{ques}{\rm\bfseries{Question}}[section]
\newtheorem{cons}[rem]{\rm\bfseries{Construction}}
\newtheorem{exm}[rem]{\rm\bfseries{Example}}

%\numberwithin{equation}{section}

      % Includes the theorem environments

%stuff from basmac.tex

\hyphenation{ho-mol-o-gous}
\newcommand\commentable[1]{#1}
\newcommand\Rk{\mathrm{rk}}
\newcommand\Id{\mathrm{Id}}
\newcommand{\lbra}{{\em (}}
\newcommand{\rbra}{{\em )}}
\newcommand{\Tors}{\mathrm{Tors}}
\newcommand\sd{\mbox{-}}
\newcommand\std{\xi_{std}}
\newcommand{\HF}{HF}
\newcommand{\OneHalf}{\frac{1}{2}}
\newcommand{\ThreeHalves}{\frac{3}{2}}
\newcommand{\OneQuarter}{\frac{1}{4}}
\newcommand{\CP}[1]{{\mathbb{CP}}^{#1}}
\newcommand{\CPbar}{{\overline{\mathbb{CP}}}^2}
\newcommand{\Zmod}[1]{\Z/{#1}\Z}

\newtheorem{question}[thm]{Question}
\newtheorem{remark}[thm]{Remark}

\newcommand{\Tr}{\mathrm{Tr}}
\newcommand{\Ker}{\mathrm{Ker}}
\newcommand{\CoKer}{\mathrm{Coker}}
\newcommand{\Coker}{\mathrm{Coker}}
\newcommand{\ind}{\mathrm{ind}}
\newcommand{\Image}{\mathrm{Im}}
\newcommand{\Span}{\mathrm{Span}}
\newcommand{\Spec}{\mathrm{Spec}}

\newcommand{\grad}{\vec\nabla}

\newcommand{\cm}{\cdot}

%stuff from macinv.tex

\newcommand\Laurent{\mathbb L}
\newcommand\renEuler{\widehat\chi}
\newcommand\spinccanf{k}
\newcommand\spinccan{\ell}

\newcommand\mCP{{\overline{\mathbb{CP}}}^2}

\newcommand\HFpRed{\HFp_{\red}}
\newcommand\HFpRedEv{\HFp_{\red, \ev}}
\newcommand\HFpRedOdd{\HFp_{\red,\odd}}
\newcommand\HFmRed{\HFm_{\red}}
\newcommand\mix{\mathrm{mix}}

\newcommand\RP[1]{{\mathbb{RP}}^{#1}}

\newcommand\liftGr{\widetilde\gr}
\newcommand\chiTrunc{\chi^{\mathrm{trunc}}}
\newcommand\chiRen{\widehat\chi}
\newcommand\Fred[1]{F^{\mathrm{red}}_{#1}}
\newcommand\SpinC{\mathrm{Spin}^c}
\newcommand{\F}{\mathbb F}
\newcommand\SpinCz{\mathfrak S}
\newcommand\interior{\mathrm{int}}

\newcommand{\ev}{\mathrm{ev}}
\newcommand{\odd}{\mathrm{odd}}
\newcommand\sRelSpinC{\underline\spinc}
\newcommand\RelSpinC{\underline{\SpinC}}
\newcommand\relspinc{\underline{\spinc}}
\newcommand\ThreeCurveComp{\Sigma-\alpha_{1}-\ldots-\alpha_{g}-\beta_{1}-\ldots-\beta_{g}-\gamma_1-...-\gamma_g}
\newcommand\SpanA{{\mathrm{Span}}([\alpha_i]_{i=1}^g)}
\newcommand\SpanB{{\mathrm{Span}}([\beta_i]_{i=1}^g)}
\newcommand\SpanC{{\mathrm{Span}}([\gamma_i]_{i=1}^g)}
\newcommand\Filt{\mathcal F}
\newcommand\HFinfty{\HFinf}
\newcommand\CFinfty{\CFinf}
\newcommand\Tai{{\mathbb T}_{\alpha}^i}
\newcommand\Tbj{{\mathbb T}_{\beta}^j}
\newcommand\RightFp{R^+}
\newcommand\LeftFp{L^+}
\newcommand\RightFinf{R^\infty}
\newcommand\LeftFinf{L^\infty}
\newcommand\Area{\mathcal A}
\newcommand\PhiIn{\phi^{\mathrm{in}}}
\newcommand\PhiOut{\phi^{\mathrm{out}}}
\newcommand\Harm{\mathcal H}
\newcommand\sS{\mathcal S}
\newcommand\sW{\mathcal W}
\newcommand\sX{\mathcal X}
\newcommand\sY{\mathcal Y}
\newcommand\sZ{\mathcal Z}
\newcommand\cent{\mathrm{cent}}

\newcommand\BigO{O}
\newcommand\ModFlowMod{\ModFlow^{*}}
\newcommand\ModCent{\ModFlow^{\mathrm{cent}}
\left({\mathbb S}\longrightarrow \Sym^{g-1}(\Sigma_{1})\times 
\Sym^2(\Sigma_{2})\right)}
\newcommand\ModSphere{\ModFlow\left({\mathbb S}\longrightarrow 
\Sym^{g-1}(\Sigma_{1})\times \Sym^2(\Sigma_{2})\right)}
\newcommand\ModSpheres\ModSphere
\newcommand\CF{CF}
\newcommand\cyl{\mathrm{cyl}}
\newcommand\CFa{\widehat{CF}}
\newcommand\CFp{\CFb}
\newcommand\CFm{\CF^-}
\newcommand\CFleq{\CF^{\leq 0}}
\newcommand\HFleq{\HF^{\leq 0}}
\newcommand\CFmeq{\CFleq}
\newcommand\CFme{\CFleq}
\newcommand\HFme{\HFleq}
\newcommand\HFpred{\HFp_{\rm red}}
\newcommand\HFpEv{\HFp_{\mathrm{ev}}}
\newcommand\HFpOdd{\HFp_{\mathrm{odd}}}
\newcommand\HFmred{\HFm_{\rm red}}
\newcommand\HFred{\HF_{\rm red}}
\newcommand\coHFm{\HF_-}
\newcommand\coHFp{\HF_+}
\newcommand\coHFinf{\HF_\infty}
\newcommand{\red}{\mathrm{red}}
\newcommand\ZFa{\widehat{ZF}}
\newcommand\ZFp{ZF^+}
\newcommand\BFp{BF^+}
\newcommand\BFinf{BF^\infty}
\newcommand\HFp{\HFb}
\newcommand\HFpm{HF^{\pm}}
\newcommand\HFm{\HF^-}
\newcommand\CFinf{CF^\infty}
\newcommand\HFinf{HF^\infty}
\newcommand\CFb{CF^+}
\newcommand\HFa{\widehat{HF}}
\newcommand\HFb{HF^+}
\newcommand\gr{\mathrm{gr}}
\newcommand\Mas{\mu}
\newcommand\UnparModSp{\widehat \ModSp}
\newcommand\UnparModFlow\UnparModSp
\newcommand\Mod\ModSp
\newcommand{\cpl}{{\mathcal C}^+}
\newcommand{\cmi}{{\mathcal C}^-}
\newcommand{\cplm}{{\mathcal C}^\pm}
\newcommand{\cald}{{\mathcal D}}

\newcommand\PD{\mathrm{PD}}
\newcommand\dist{\mathrm{dist}}
\newcommand\Paths{\mathcal{B}}
\newcommand\Met{\mathfrak{Met}}
\newcommand\Lk{\mathrm{Lk}}
\newcommand\Jac{\mathfrak{J}}
\newcommand\spin{\mathfrak s}

\newcommand{\xTuple}{\underline x}
\newcommand{\yTuple}{\underline y}
\newcommand{\aTuple}{\underline a}
\newcommand{\bTuple}{\underline b}
\newcommand{\xiTuple}{\underline \xi}
\newcommand{\spinc}{\mathfrak s}
\newcommand{\spincX}{\mathfrak r}
\newcommand{\spinct}{\mathfrak t}
\newcommand{\Exp}{\mathrm{Exp}}
\newcommand{\etaTuple}{\underline \eta}

\newcommand\Real{\mathrm Re}
\newcommand\brD{F}
\newcommand\brDisk{\brD}
\newcommand\brPhi{\widetilde\phi}
\newcommand\Perm[1]{S_{#1}}
\newcommand\BranchData{\mathrm Br}
\newcommand\ModMaps{\mathcal M}
\newcommand\ModSp\ModMaps
\newcommand\ModMapsComp{\overline\ModMaps}
\newcommand\tsModMaps{\ModMaps^\circ}
\newcommand\ModCurves{\mathfrak M}
\newcommand\ModCurvesComp{\overline\ModCurves}
\newcommand\ModCurvesSmoothInt{\ModCurves_{si}}
\newcommand\ModCurvesSmoothBoundary{\ModCurves_{sb}}
\newcommand\BranchMap{\Phi}
\newcommand\SigmaMap{\Psi}
\newcommand\ProdSig[1]{\Sigma^{\times{#1}}}
\newcommand\ProductForm{\omega_0}
\newcommand\ProdForm{\ProductForm}
\newcommand\Energy{E}

\newcommand\qTup{\mathfrak q}
\newcommand\pTup{\mathfrak p}

\newcommand\Cinfty{C^{\infty}}
\newcommand\Ta{{\mathbb T}_{\alpha}}
\newcommand\Tb{{\mathbb T}_{\beta}}
\newcommand\Tc{{\mathbb T}_{\gamma}}
\newcommand\Td{{\mathbb T}_{\delta}}
\newcommand\Na{N_{\alpha}}
\newcommand\Nb{N_{\beta}}
\newcommand\Torus{\mathbb{T}}
\newcommand\Diff{\mathrm{Diff}}

\newcommand\dbar{\overline\partial}
\newcommand\Map{\mathrm{Map}}
\newcommand\Strip{\mathbb{D}}

\newcommand\mpa{p}
\newcommand\mpb{q}
\newcommand\Mac{M}

\newcommand\del{\partial}

\newcommand\alphas{\mbox{\boldmath$\alpha$}}
\newcommand\mus{\mbox{\boldmath$\mu$}}
\newcommand\xis{\mbox{\boldmath$\xi$}}
\newcommand\etas{\mathbf\eta}
\newcommand\betas{\mbox{\boldmath$\beta$}}
\newcommand\gammas{\mbox{\boldmath$\gamma$}}
\newcommand\deltas{\mbox{\boldmath$\delta$}}
\newcommand\HFto{HF _{{\rm to}} ^{SW}}
\newcommand\CFto{CF _{{\rm to}} ^{SW}}
\newcommand\HFfrom{HF_{{\rm from}} ^{SW}}
\newcommand\CFfrom{CF_{{\rm from}} ^{SW}}

\newcommand\NumPts{n}
\newcommand\Ann{\mathrm{Ann}}

\newcommand\Dom{\mathcal D}
\newcommand\PerClass[1]{{\PerDom}_{#1}}
\newcommand\PerClasses[1]{{\Pi}_{#1}}
\newcommand\PerDom{\mathcal P}
\newcommand\RenPerDom{\mathcal Q}
\newcommand\intPerDom{\mathrm{int}\PerDom}
\newcommand\csum{*}
\newcommand\CurveComp{\Sigma-\alpha_{1}-\ldots-\alpha_{g}-\beta_{1}-\ldots-\beta_{g}}
\newcommand\CompA{\Sigma-\alpha_{1}-\ldots-\alpha_{g}}
\newcommand\CompB{\Sigma-\beta_{1}-\ldots-\beta_{g}}
\newcommand\ta{t_\alpha}
\newcommand\tb{t_\beta}
\newcommand\EmbSurf{Z}
\newcommand\Mult{m}

\newcommand\NumDoms{m}

\newcommand\uCF{\underline\CF}
\newcommand\uCFinf{\uCF^\infty}
\newcommand\uHF{\underline{\HF}}
\newcommand\uHFp{\underline{\HF}^+}
\newcommand\uHFmred{\underline{\HF}^-_{\red}}
\newcommand\uHFpRed{\uHFpred}
\newcommand\uHFpred{\underline{\HF}^+_{\red}}
\newcommand\uCFp{\underline{\CFp}}
\newcommand\uHFm{\underline{\HF}^-}
\newcommand\uHFa{\underline{\HFa}}
\newcommand\uCFa{\underline{\CFa}}
\newcommand\uDel{\underline{\partial}}
\newcommand\uHFinf{\uHF^\infty}

\newcommand\alphaperps{\{\alpha_1^\perp,...,\alpha_g^\perp\}}
\newcommand\betaperps{\{\beta_1^\perp,...,\beta_g^\perp\}}
\newcommand\gammaperps{\{\gamma_1^\perp,...,\gamma_g^\perp\}}
\newcommand\ufinf{\underline{f^\infty_{\alpha,\beta,\gamma}}}
\newcommand\uHFleq{\underline{\HFleq}}
\newcommand\uCFleq{\underline{\CFleq}}
\newcommand\ufp[1]{\underline{f^+_{#1}}}
\newcommand\ufleq[1]{\underline{f^{\leq 0}_{#1}}}
\newcommand\ufa{\underline{{\widehat f}_{\alpha,\beta,\gamma}}}
\newcommand\uFa[1]{\underline{{\widehat F}_{#1}}}
\newcommand\uFm[1]{{\underline{F}}^-_{#1}}
\newcommand\Fstar[1]{F^{*}_{#1}}
\newcommand\uFp[1]{{\underline{F^+}}_{#1}}
\newcommand\uFinf[1]{{\underline F}^\infty_{#1}}
\newcommand\uFleq[1]{{\underline F}^{\leq 0}_{#1}}
\newcommand\uFstar[1]{{\underline F}^{*}_{#1}}
\newcommand\Fm[1]{F^{-}_{#1}}
\newcommand\FmRed[1]{F^{-}_{#1}}
\newcommand\Fp[1]{F^{+}_{#1}}
\newcommand\Fpm[1]{F^{\pm}_{#1}}
\newcommand\Fa[1]{{\widehat F}_{#1}}
\newcommand\Finf[1]{F^{\infty}_{#1}}
\newcommand\Fleq[1]{F^{\leq 0}_{#1}}
\newcommand\Tx{{\mathbb T}_\xi}
\newcommand\Ty{{\mathbb T}_\eta}
\newcommand\Ring{\mathbb A}
\newcommand\HFstar{\HF^*}
\newcommand\uLeftFinf{{\underline L}^\infty}
\newcommand\uRightFinf{{\underline R}^\infty}
\newcommand\orient{\mathfrak o}

%stuff from macnew.tex

\newcommand\essarc{\boldmath{E}}
\newcommand\essarcP{\essarc_+}
\newcommand\essarcM{\essarc_-}

\newcommand\MarkPt{p}

\newcommand\CFKlgeq{CFK^{<,\geq}}
\newcommand\spincrel\relspinc
\newcommand\relspinct{\underline{\mathfrak t}}

\newcommand\CFKsa{CFK^0}
\newcommand\bCFKsa{^bCFK^0}
\newcommand\spincf{\mathfrak r}

\newcommand\Mass{\overline n}
\newcommand\Vertices{\mathrm{Vert}}

\newcommand\CFB{CFB}
\newcommand\CFBinf{\CFB^\infty}
\newcommand\Bouq{\mathbb B}
\newcommand\zees{\mathbf z}

\newcommand\MCone{M}
\newcommand\bCFKp{{}^b\CFKp}
\newcommand\bCFKm{{}^b\CFKm}
\newcommand\bPhip{{}^b\Phi^+}
\newcommand\bPhim{{}^b\Phi^-}
\newcommand\Phip{\Phi^+}
\newcommand\Phim{\Phi^-}
\newcommand\bCFKa{{}^b\CFKa}
\newcommand\CFK{CFK}
\newcommand\HFK{HFK}

\newcommand\CFKc{\CFK^\circ}
\newcommand\CFKp{\CFK^+}
\newcommand\CFKa{\widehat\CFK}
\newcommand\CFKm{\CFK^-}
\newcommand\CFKinf{\CFK^{\infty}}
\newcommand\HFKp{\HFK^+}
\newcommand\bHFKp{{}^b\HFK^+}
\newcommand\bHFKm{{}^b\HFK^-}
\newcommand\HFKa{\widehat\HFK}
\newcommand\HFKm{\HFK^-}
\newcommand\HFKinf{\HFK^{\infty}}
\newcommand\HFKc{\HFK^\circ}

\newcommand\Mark{m}
\newcommand\Marks{\mathbf{m}}

\newcommand\BasePt{w}
\newcommand\FiltPt{z}

\newcommand\StartPt{w}
\newcommand\EndPt{z}
\newcommand\StartPts{\mathbf{w}}
\newcommand\EndPts{\mathbf{z}}

\newcommand\is{\mathbf{i}}
\newcommand\js{\mathbf{j}}

\newcommand\bPsip{{}^b\Psi^+}
\newcommand\bPsia{{}^b{\widehat\Psi}}
\newcommand\bPsim{{}^b\Psi^-}
\newcommand\Psip{\Psi^+}
\newcommand\Psim{\Psi^-}

%stuff from macf.tex
\newcommand{\divis}{\mathfrak d}
\newcommand\Dehn{D}
\newcommand\PiRed{\Pi^\red}

\newcommand\liftalpha{\widetilde\alpha}
\newcommand\liftbeta{\widetilde\beta}
\newcommand\sj{\mathfrak j}
\newcommand\MT{t}
\newcommand\Jacobi[2]{\left(\begin{array}{c} #1 \\ #2 \end{array}\right)}
\newcommand\Cobord{{\mathfrak C}(3)}
\newcommand\KOver{K_+}
\newcommand\KUnder{K_-}
\newcommand\Fiber{F}
\newcommand\prefspinc{\mathfrak u}
\newcommand\spincu{\mathfrak u}
\newcommand\SpinCCobord{\theta^c}
\newcommand\Field{\mathbb F}
\newcommand\Fmix[1]{F^{\mix}_{#1}}
\newcommand\ugc[1]{{\underline g}^\circ_{#1}}
\newcommand\uGc[1]{{\underline G}^\circ_{#1}}
\newcommand\uginf[1]{{\underline g}^\infty_{#1}}
\newcommand\uGinf[1]{{\underline G}^\infty_{#1}}
\newcommand\uec[1]{{\underline e}^\circ_{#1}}
\newcommand\uEc[1]{{\underline E}^\circ_{#1}}
\newcommand\Ec[1]{{E}^\circ_{#1}}
\newcommand\Gc[1]{{G}^\circ_{#1}}
\newcommand\ueinf[1]{{\underline e}^\infty_{#1}}
\newcommand\uEinf[1]{{\underline E}^\infty_{#1}}
\newcommand\rspinc{\underline\spinc}
\newcommand\Fc{F^\circ}
\newcommand\fc{f^\circ}
\newcommand\uFc{{\underline F}^\circ}
\newcommand\ufc{{\underline f}^\circ}
\newcommand\uF{{\underline F}^\circ}
\newcommand\uHFc{\uHF^\circ}
\newcommand\uCFc{\uCF^\circ}
\newcommand\HFc{\HF^\circ}
\newcommand\CFc{\CF^\circ}

\newcommand\Dual{\mathcal D}
\newcommand\Duality\Dual
\newcommand\aPrBd{\partail_{\widetilde \alpha}}
\newcommand\bPrBd{\partail_{\widetilde \beta}}
\newcommand\cPrBd{\partail_{\widetilde \gamma}}

\newcommand\liftTa{{\widetilde{\mathbb T}}_\alpha}
\newcommand\liftTb{{\widetilde{\mathbb T}}_\beta}
\newcommand\liftTc{{\widetilde{\mathbb T}}_\gamma}
\newcommand\liftTd{{\widetilde{\mathbb T}}_\delta}
\newcommand\liftTcPr{{\widetilde{\mathbb T}}_\gamma'}

\newcommand\Tor{\mathrm{Tor}}
\newcommand\TaPr{\Ta'}
\newcommand\TbPr{\Tb'}
\newcommand\TcPr{\Tc'}
\newcommand\TdPr{\Td'}
\newcommand\alphaprs{\alphas'}
\newcommand\betaprs{\betas'}
\newcommand\gammaprs{\gammas'}
\newcommand\Spider{\sigma}
\newcommand\EulerMeasure{\widehat\chi}

\newcommand\RelInv{F}

\newcommand\Knot{\mathbb K}
\newcommand\Link{\mathbb L}
\newcommand\SubLink{\mathbb K}
\newcommand\SublinkC{{\SubLink}'}
\newcommand\liftSym{{\widetilde \Sym}^g(\Sigma)_\xi}
\newcommand\liftCFinf{\widetilde{CF}^\infty}
\newcommand\liftz{\widetilde z}
\newcommand\liftCFp{\widetilde{CF}^+}
\newcommand\liftCF{\widetilde{CF}}
\newcommand\liftHFinf{\widetilde{HF}^\infty}
\newcommand\liftHFp{\widetilde{HF}^+}
\newcommand\liftx{\widetilde\x}
\newcommand\lifty{\widetilde\y}
\newcommand\liftSigma{\widetilde\Sigma}
\newcommand\liftalphas{\widetilde \alphas}
\newcommand\liftbetas{\widetilde \betas}
\newcommand\liftgammas{{\widetilde\gammas}}
\newcommand\liftdeltas{{\widetilde\deltas}}
\newcommand\liftgammasPr{\liftgammas'}
\newcommand\liftGamma{\widetilde{\Gamma}}
\newcommand\liftdeltasPr{\liftdeltas'}
\newcommand\Tz{{\widetilde{\mathbb T}}_\zeta}
\newcommand\liftTheta{\widetilde\Theta}

\newcommand\liftF[1]{{\widetilde F}_{#1}}
\newcommand\spinck{\mathfrak k}

\newcommand\FiltY{\Filt_\spinc}

\newcommand\Pos{\mathcal{P}}

\newcommand\ons{Ozsv{\'a}th and Szab{\'o}}
\newcommand\os{{Ozsv{\'a}th-Szab{\'o}}}

%==========================================
%% Do not edit the following command
\setcounter{page}{1}
\volume{13}
%==========================================

\title[{Some remarks on  cabling, contact structures, and complex curves}] 
{Some remarks on cabling, contact structures, and complex curves}

\author[Matthew Hedden]{Matthew Hedden}
\address{Department of
Mathematics, Massachusetts Institute of Technology, MA}
\email{mhedden@math.mit.edu}

\begin{abstract}
We determine the relationship between the contact structure induced by a fibered knot, $K\subset S^3$, and the contact structures induced by its various cables.  Understanding this relationship allows us to classify  fibered cable knots which bound a properly embedded complex curve in the four-ball satisfying a genus constraint. This generalizes the well-known classification of links of plane curve singularities.
\end{abstract}
\keywords{contact structure, fibered knot, cable knot, complex curve, link of singularity}

\maketitle
\section{Introduction}
A well-known construction of Thurston and Winkelnkemper \cite{Thurston} associates a contact structure to an open book decomposition of a three-manifold.  This allows us to talk about {\em the contact structure associated to a fibered knot}.  Here, a fibered knot is a pair, $(F,K)\subset Y$, such that $Y-K$ admits the structure of a fiber bundle over the circle with fibers isotopic to $F$ and $\partial F=K$.  We denote the contact structure associate to a fibered knot by $\xi_{F,K}$ or, when the fiber surface is unambiguous, by $\xi_{K}$.

Thus any operation on knots (or Seifert surfaces) which preserves the property of fiberedness induces an operation on contact structures.  For instance, one can consider the Murasugi sum operation on surfaces-with-boundary (see \cite{Gabai} for definition).  In this case, a result of Stallings  \cite{Stallings} indicates that the Murasugi sum of two fiber surfaces is also fibered (a converse to this was proved by Gabai \cite{Gabai}). The effect on contact structures is given by a result of Torisu:
\begin{thm}{\em(}Theorem $1.3$ of \cite{Torisu}.{\em)} Let $(F_1,\partial F_1)\subset Y_1$ $(F_2,\partial F_2)\subset Y_2$ be two fiber surfaces, and let $(F_1*F_2, \partial (F_1*F_2))\subset Y_1\#Y_2$ denote any Murasugi sum.  Then $$ \xi_{F_1*F_2, \partial (F_1*F_2)}\simeq \xi_{F_1,\partial F_1}\# \xi_{F_2,\partial F_2},$$
\noindent where the right-hand side denotes the connected sum of contact structures.  
\end{thm}

Another operation which preserves fiberedness is cabling, whose definition we now recall.   For a knot $K\subset Y$,  there is an identification of a tubular neighborhood of $K$ with a solid torus, $S^1\times D^2$, and a {\em cable knot} is the image of a torus knot living on $\partial (S^1\times D^2)$ under this identification. Note that the identification depends on a choice of longitude for $K$, and we choose the canonical longitude coming from a Seifert surface for K.  We denote the $(p,q)$ cable of $K$ by $K_{p,q}$, with $q$ indicating the linking number of $K_{p,q}$ with $K$ i.e. the number of times the torus knot wraps around the meridian of the tubular neighborhood of $K$.  Thus $U_{p,q}=T_{p,q}$ where $U$ is the unknot and $T_{p,q}$ is the $(p,q)$ torus knot.  Throughout we will consider knots, and hence $p$ and $q$ will be relatively prime.  We also restrict to the case $p>0$.  There is no loss of information in doing this, since $K_{p,q}=K_{-p,-q}$ as unoriented knots, and none of our arguments are sensitive to orientation.   

In the case now that $K$ is fibered, $K_{p,q}$ will also be fibered.   This follows, for instance, from \cite{Stallings} though the fibration of $Y-K_{p,q}$ can easily be visualized.   Thus it makes sense to ask how the contact structure associated to a fibered knot changes, if at all, upon cabling.   The purpose of this note is to  answer this question for the case of knots in $S^3$.  To state the result, recall that contact structures on $S^3$ are in bijection with the integers plus a distinguished element \cite{Eliashberg,Eliashberg2}:
 $$\{ \text{Isotopy\ classes\ of\ contact\ structures\ on\ }S^3\} \leftrightarrow \Z \cup \{ \bullet\}.$$
\noindent  The contact structure corresponding to the distinguished element, which we denote $\std$, is the unique tight contact structure.  It is obtained as the field of complex lines tangent to the three-sphere, viewed as the boundary of the Stein four-ball in $\C^2$.  The other, overtwisted, contact structures correspond to homotopy classes of two-plane fields on $S^3$ \cite{Eliashberg}, and are determined uniquely by their Hopf invariant, $h(\xi)\in \Z$.  We let $\xi_i$ denote the overtwisted contact structure with $h(\xi_i)=i$.  In these terms, we have:
\begin{thm} \label{thm:main}  Let $K\subset S^3$ be a fibered knot of Seifert genus $g$, and let $K_{p,q}$ denote its $(p,q)$ cable. Then for any $p>0$, we have 
\begin{enumerate} 
\item If $q>0$, then $\xi_{K_{p,q} }\simeq \xi_K.$
\item If $q<0$, then $\xi_{K_{p,q}} \simeq \xi_K\# \xi_{(1-p)(2g-q-1)}$
\end{enumerate}
\noindent In particular,  $\xi_{K_{p,q}}\simeq \std$ if and only if $\xi_K\simeq \std$ and $q>0$.

\end{thm}

Cable knots play a prominent role in the classification of plane curve singularities, where certain iterated torus knots naturally arise (an iterated torus knot is an iterated cable of the unknot).  To explain this, let $f(z,w)\in \C[z,w]$ be a complex polynomial with isolated singularity at the origin $(0,0)\in \C^2$.  Recall that the {\em link of the singularity of $f(z,w)$} is defined to be the knot or link, $K_f$, which comes from the intersection 
 $$ K_f= \{ (z,w)\in \C^2 | f(z,w)=0 \} \cap S^3_\epsilon,$$
 
\noindent where $S^3_\epsilon = \{(z,w)\in \C^2| |z|^2+|w|^2=\epsilon \}$ is the unit sphere of radius $\epsilon$ in $\C^2$.
It is straightforward to show  that the isotopy type of $K_f$ as a knot or link is well-defined, provided $\epsilon$ is sufficiently small.  The classification of links of plane curve singularities shows that the only knots which arise in this way are iterated cables of the unknot whose cabling parameters satisfy a positivity condition (see \cite{Eisenbud} for a discussion).  For example, the $(p',q')$ cable of the $(p,q)$ torus knot is the link of a singularity precisely when $q'\ge p'pq+1$.   

An important property of links of singularities is that they bound smooth complex curves of genus equal to their Seifert genus.  Indeed, a small perturbation of $f$ produces a smooth complex curve which can be deformed into the three-sphere without singularities.  This surface is a Seifert surface for $K_f$ which can be seen to be a fiber in a fibration of $S^3-K_f$ (see \cite{Milnor}  for more details).   

While the iterated torus knots which arise as links of singularities are classified, one can relax the condition that the radius of the three-sphere $S^3_\epsilon$ be sufficiently small. In this case, the four-ball may contain multiple singular points of a given complex curve.  The general case is handled by fixing the radius of the three-sphere and asking:

\begin{question} Which iterated torus knots arise as the intersection of a complex curve  $V_f\subset \C^2$ with $S^3_1$, the unit radius three-sphere? Of these, which among them have a Seifert surface which can be isotoped into the four-ball to 
the piece of the curve, $V_f\cap B^4$?
\end{question}

Drawing on our previous work \cite{SQPfiber}, Theorem \ref{thm:main} answers the second part of the question as a corollary.

\begin{cor} \label{cor:complex}
Let $p>0$, and let $K$ be a fibered knot. Then $K_{p,q}$ has a Seifert surface which is isotopic to a piece of a complex curve $V_f \cap B^4$ if and only if
\begin{itemize}
\item $K$ has a  Seifert surface which is isotopic to a piece of a complex curve {\em and}
\item $q>0$
\end{itemize}
\noindent In particular, the fiber surface of an iterated torus knot is isotopic to a piece of a complex curve if and only if all the cabling coefficients are positive.
\end{cor}

\begin{remark}  The corollary answers a question of  Baader and Ishikawa (Question $6.10$ of \cite{Baader}).  They asked whether there exists a quasipositive fiber surface, other than a disk, from which we cannot deplumb a Hopf band (see Figure \ref{fig:SQPcable} and the associated discussion in Section \ref{sec:proof} for relevant definitions).  The answer is yes, and an example is provided by the $(2,1)$ cable of the right-handed trefoil.  The proof of the corollary shows that  it bounds a quasipositive fiber surface, but it does not deplumb a Hopf band by work of Melvin and Morton \cite{Melvin}.
\end{remark}

We prove Theorem \ref{thm:main} and Corollary \ref{cor:complex} in the next section.  Theorem \ref{thm:main} will be proved by examining the knot Floer homology invariants of cable knots, which we studied in \cite{Cable,CableII}.  For fibered knots in the three-sphere, these invariants completely determine the contact structure induced by the knot.   Corollary \ref{cor:complex} will follow from a result in \cite{SQPfiber} which connects a particular Floer invariant or, equivalently, the tightness of the contact structure induced by $K$, with complex curves.  

While Theorem \ref{thm:main} cannot be extended in general using the techniques of the present paper, it can be done in the case of three-manifolds, $Y^3$, for which the \os \ contact invariant classifies contact structures on $Y^3$ (e.g. lens spaces).  However, it seems reasonable to expect that the theorem holds for arbitrary three-manifolds.   That is, positive cabling should not change the contact structure associated to a fibered knot, while negative cabling has the effect of taking the connected sum with $(S^3,\xi_{(1-p)(2g-q-1)})$.  Indeed, this could probably be proved with the rudiments of convex surface theory.  We content ourselves here with the case of $S^3$ since the results follow easily from our previous work on cables. Moreover, our proof serves to highlight the rich geometric content of the knot Floer homology invariants.  In particular, we found it surprising that knot Floer homology could be useful in a result such as Corollary \ref{cor:complex}.  

\bigskip

\noindent {\bf Acknowledgment:} This note draws on several aspects of a talk I gave  at the 2007 G{\"o}kova Geometry and Topology conference.  I wish to thank the organizers of the conference for their invitation to speak and attend, and the  NSF and IMU for providing funding.  In addition, I thank John Etnyre for his interest in this work.   Partial funding was provided for by NSF DMS-0706979.

\section{Proof of Theorem \ref{thm:main}}
\label{sec:proof}
In this section we prove Theorem \ref{thm:main}.  This is accomplished in two steps.   The first is to determine the effect of cabling on the Hopf invariant.   The next is to determine  whether a cable knot induces the tight contact structure, $\std$.  Since contact structures on $S^3$ are determined by their Hopf invariant and whether they are tight, this will complete the proof.   Both tasks can be accomplished with an understanding of the knot Floer homology invariants of cable knots, which we established in \cite{Cable,CableII}.   The knot Floer homology invariants of a fibered knot $K\subset S^3$, in turn, determine the Hopf invariant of $\xi_K$ and whether $\xi_K$ is tight.

\subsection{Contact structures and knot Floer homology of fibered knots}
We begin by briefly describing the algebraic structure of the knot Floer homology invariants and how these invariants determine the contact structure associated to a fibered knot $K\subset S^3$.

%The key observation will be that the knot Floer homology invariants of a fibered knot $K\subset S^3$ completely determine the contact structure $\xi_K$.  From this observation, Theorem \ref{thm:main} will follow from our previous work \cite{Cable,CableII} which determines the effect of cabling on the knot Floer homology invariants in the case when $|q|\gg 0$, together with an observation of Rudolph \cite{Rudolph} which relates cables with $|q|\gg 0 $ to cables with $|q|$ arbitrary.  To explain this connection between contact structures and knot Floer homology, we must first recall some basics about the structure of knot Floer homology.

To a closed oriented three-manifold, $Y$, and Spin$^c$ structure, $\spinc$,  \ons \  introduced a chain complex $\CFa(Y,\spinc)$ \cite{HolDisk}.  A null-homologous knot $(Y,K)$ induces a filtration $\Filt(Y,K,i)$ of this chain complex, i.e. there is an increasing sequence of subcomplexes:
$$ 0=\Filt(Y,K,-i) \subseteq \Filt(Y,K,-i+1)\subseteq \ldots \subseteq
\Filt(Y,K,n)=\CFa(Y,\spinc).$$
	
\noindent For the definition of this filtration see \cite{Knots,Ras1}. The homology of the successive quotients in this filtration are the {\em knot Floer homology groups of $(Y,K)$}.  In the case of the three-sphere, we denote these groups by
$$\HFKa_*(K,i)\cong H_*\left(\frac{\Filt(S^3,K,i)}{\Filt(S^3,K,i-1)}\right)$$
 
The knot Floer homology groups have the following symmetries (Propositions $3.7$ and $3.10$ of \cite{Knots}, respectively):
\begin{equation}
\label{eq:mirror}
\HFKa_{*}(K,i)\cong \HFKa_{-*}(\overline{K},-i).
\end{equation}
 \begin{equation}
\label{eq:jsym}
\HFKa_{*}(K,i)\cong \HFKa_{*-2i}(K,-i).
\end{equation}
 
 In the first equation, $\overline{K}$ denotes the reflection of $K$ (i.e. in a projection, change all crossings simultaneously or, equivalently, consider the image of $K$ under an orientation-reversing involution of $S^3$).  In \cite{Contact}, \ons \ proved the following:
\begin{prop}\label{prop:hopf}
Let $K\subset S^3$ be a fibered knot, and let $g(K)=g$ denote its Seifert genus.  Then
$$\HFKa_*(K,g)\cong 
\left\{\begin{array}{ll}

 $\Z$ & {*=-h(\xi_K)} \\
 
0 & {\text{otherwise.}}\\
\end{array}
\right. $$
\noindent where $h(\xi_K)$ is the Hopf invariant of the contact structure induced by $K$.
\end{prop}
\begin{remark} This proposition is a combination of Theorem $1.1$ and Proposition $4.6$ of \cite{Contact}.   More precisely, Theorem $1.1$ shows that $$\HFKa_*(-S^3,K,-g)\cong H_*(\Filt(-S^3,K,-g))\cong \Z,$$ \noindent and Proposition $4.6$ shows that the grading of this group is $h(\xi_K)$.  Equation \eqref{eq:mirror} then determines the effect of reflection on knot Floer homology (i.e. $(-S^3,K)=(S^3,\overline{K})$).  Note that each of the aforementioned results generalize to knots in manifolds other than $S^3$.
\end{remark}

This yields the following  property of $h(\xi_K)$, which can also be found in \cite{Rudolph3}.
\begin{cor}\label{cor:hopfmirror}
The Hopf invariant satisfies $h(\xi_K)=-h(\xi_{\overline{K}})-2g$, where $g$ is the genus of $K$.   
\end{cor}
\begin{proof} Combine Proposition \ref{prop:hopf} with Equations \eqref{eq:mirror} and \eqref{eq:jsym}.\end{proof}

The above proposition indicates that the grading on knot Floer homology captures the homotopy class of the contact structure.  In the case of the three-sphere, it remains to understand whether a contact structure is tight or overtwisted.  To this end, recall that the Floer homology of the three-sphere satisfies
$$ \HFa_*(S^3)\cong 
\left\{\begin{array}{ll}

 $\Z$ & {*=0} \\
 
0 & {\text{otherwise.}}\\
\end{array}
\right. $$
\noindent Thus, we can define the following numerical invariant of $K$,\cite{FourBall,Ras1}:
$$\tau(K)=\mathrm{min}\{j\in\Z|i_* : H_*(\Filt(S^3,K,j))\longrightarrow \HFa(S^3)\cong \Z \  \mathrm{is \ non\-trivial}\}.$$

For fibered knots, $\tau(K)$ determines whether $\xi_K$ is tight or overtwisted.  Moreover, it determines whether $K$ arises as the boundary of a properly embedded complex curve $V_f\subset B^4$ satisfying $g(V_f)=g(K)$. 

\begin{prop} \lbra Proposition $2.1$ of \cite{SQPfiber}.\rbra \label{prop:SQPfiber}
	Let $K\subset S^3$ be a fibered knot and $F$ its fiber surface. Then the following are equivalent:
	\begin{enumerate}
		\item $K$ satisfies $g(K)=\tau(K).$
		\item  $\xi_K\simeq \std$.
		\item $K$ is strongly quasipositive \lbra with $F$ a quasipositive Seifert surface \rbra.
	
		\item  $F$ is isotopic to a properly embedded complex curve $(V_f,K) \subset (B^4,S^3)$.
	\end{enumerate}
\end{prop}
\begin{remark} A strongly quasipositive knot is a knot which bounds a so-called quasipositive Seifert surface.  Quasipositive Seifert surfaces, in turn, are those surfaces-with-boundary which are obtained from parallel disks by attaching positive bands.  See Figure \ref{fig:SQPcable} for examples, and \cite{SQPfiber} for further details.  We have added $(4)$ above, to the equivalences of \cite{SQPfiber}.  However, $(3)$ implies $(4)$ by work of Rudolph \cite{Rudolph}, and $(4)$ implies $(1)$ by Theorem $1.5$ of \cite{SQPfiber}.
\end{remark}

Thus, to determine the effect of cabling on contact structures, it will suffice to understand the behavior of the grading of the ``top group" of knot Floer homology under cabling, together with the behavior of $\tau(K)$.  We begin with the former.

\subsection{Behavior of the Hopf invariant under cabling}

In this section, we determine the effect of cabling on the Hopf invariant of a fibered knot.  We prove the following:

\begin{thm}\label{thm:cablehopf}

Let $p>0$, and let $K$ be a fibered knot of genus $g$.  Then we have
$$h(\xi_{K_{p,q}})=
\left\{\begin{array}{ll}
h(\xi_K) & {\text{for\ }  q>0} \\
h(\xi_K)+(1-p)(2g-q-1) & {\text{for\ }  q<0}.
\end{array}
\right.$$
\end{thm}
\begin{remark}  This result can be derived from a theorem of Neumann and Rudolph \cite{IndexDiff}, proved in a different context.  They proved a similar formula for the {\em enhanced Milnor number} of a fibered link, which they denote $\lambda(K)$. Our formula follows from theirs since, after tracing through the definitions of the Hopf invariant and enhanced Milnor number, one finds that $\lambda(K)=-h(\xi_K)$.  Our proof, using Floer homology, is quite different in spirit and follows easily from our previous understanding of the Floer homology of cable knots. 
\end{remark}

The main tool will be the following result from \cite{Cable}
\begin{prop} \lbra Corollary $1.3$ of \cite{Cable}.\rbra 
\label{prop:largen}
Let $p>0$ and
let $K\subset S^3$ be a knot of Seifert genus $g$. Then there exists an $N>0$ such that for all  $n>N$, the following holds:  

$$ \HFKa_*\left(K_{p,pn+1},\ pg+\frac{(p-1)(pn)}{2}\right) \cong 
\HFKa_*(K,g)$$

\end{prop}
\begin{remark}
Note that $pg+\frac{(p-1)(pn)}{2}$ is the genus of $K_{p,pn+1}$.  
\end{remark}
Propositions \ref{prop:hopf} and \ref{prop:largen}  determine the effect of cabling on the Hopf invariant when $q=pn+1\gg0$.  An observation of Neumann and Rudolph \cite{Unfoldings} which interprets an arbitrary cable as a Murasugi sum of the $(p,\pm 1)$ cable and an appropriate torus knot will allow us to take care of the case when $q>0$.

\begin{prop}\label{prop:msum} \lbra Figure $4.2$ of \cite{Unfoldings} or Figure $1$ of \cite{Rudolph2}.\rbra  
\ Assume $p>0$.  Then $K_{p,q}$ is a Murasugi sum of $K_{p,\mathrm{sgn}(q)}$ and $T_{p,q}$.  Here, $\mathrm{sgn}(q)$ is the function which is $1$ if $q>0$ and $-1$ otherwise.
\end{prop}
\noindent In the above, the Murasugi sum can be taken to be along minimal genus Seifert surfaces which, for our purposes, will be fiber surfaces.

%For the case $q>0$, Proposition \ref{prop:largen} tells us that $\HFKa(K_{p,pn+1}, g(K_{p,pn+1})\cong \HFKa(K,g(K))$ whenever $n>0$ is sufficiently large.  Prop \ref{prop:hopf} then tells us that the grading of this group is negative the Hopf invariant, so that $h(\xi_{K_{p,pn+1}})=h(\xi_K)$ whenever $n\gg0$.  The above proposition then shows that $K_{p,pn+1}=K_{p,1}*T_{p,pn+1}$, where $T_{p,pn+1}$ is the $(p,pn+1)$ torus knot.   The Hopf invariant of $T_{p,pn+1}$ (or any positive torus knot) is easily seen to be $0$ (it is a plumbing of positive Hopf bands, for instance).  Torisu's theorem mentioned in the introduction then implies that the $$h(\xi_{K_{p,1})=h(\xi_{K_{p,pn+1})+h(\xi_{T_{p,pn+1})=h(\xi_{K_{p,pn+1}).$$
%Another application of Proposition \ref{prop:msum} shows that $h(\xi_{K_{p,q})=h(\xi_{K_{p,1})$ for all $q>0$, proving that $h(\xi_{K_{p,q})=h(\xi_{K})$ whenever $q>0$.

\bigskip

\noindent {\bf Proof of Theorem \ref{thm:cablehopf} when $q>0$:}

For all sufficiently large $n>0$, we have
$$h(\xi_K)=h(\xi_{K_{p,pn+1}})=h(\xi_{K_{p,1}*T_{p,pn+1}})=h(\xi_{K_{p,1}})+h(\xi_{T_{p,pn+1}})=$$
$$=h(\xi_{K_{p,1}})+h(\xi_{T_{p,q}})=h(\xi_{K_{p,1}*T_{p,q}})=h(\xi_{K_{p,q}}).$$
\noindent Here, the first equality follows from Proposition \ref{prop:largen}, the second from Proposition \ref{prop:msum}, and the third from the well-known additivity of the Hopf invariant under Murasugi sum (see, for instance \cite{Unfoldings}).  The fourth equality follows from the fact that any positive torus knot has Hopf invariant $0$ (which can be seen, for instance by realizing them all as plumbings of positive Hopf bands).  The fifth and six equalities are an application of additivity of Hopf invariant and Proposition \ref{prop:msum}, respectively.   
$\Box$

\bigskip

\noindent {\bf Proof of Theorem \ref{thm:cablehopf} when $q<0$:}

For all sufficiently large $n>0$, we have
 \begin{equation} \label{eq:string1} \begin{array}{c} h(\xi_{{K_{p,-pn-1}}})=-h(\xi_{\overline{K}_{p,pn+1}})-2\left(pg+\frac{(pn)(p-1)}{2}\right)=\\  \\ = -h(\xi_{\overline{K}})-2pg+(1-p)pn= h(\xi_K)+(1-p)(2g+pn). \end{array} \end{equation}
\noindent The first equality follows from Corollary \ref{cor:hopfmirror}, bearing in mind the genus of $K_{p,-pn-1}$ and the fact that $\overline{K_{p,-pn-1}}=\overline{K}_{p,pn+1}$. The second equality follows from the previous case, since  $pn+1>0$. The third equality is an application of Corollary \ref{cor:hopfmirror}.  

Next we have
\begin{equation} 
\label{eq:string2} h(\xi_{K_{p,-pn-1}})=h(\xi_{K_{p,-1}})+h(\xi_{T_{p,-pn-1}})=h(\xi_{K_{p,-1}})+(1-p)(pn). \end{equation}

\noindent Here, the first equality comes from Proposition \ref{prop:msum} and additivity of the Hopf invariant under Murasugi sums.  The second equality comes from Corollary \ref{cor:hopfmirror}, noting that  $\overline{T_{p,-pn-1}}=T_{p,pn+1}$ and $2g(T_{p,-pn-1})=(p-1)pn$.

By the same considerations, we have 
\begin{equation} \label{eq:string3} h(\xi_{K_{p,q}})=h(\xi_{K_{p,-1}})+h(\xi_{T_{p,q}})=h(\xi_{K_{p,-1}})+(1-p)(-q-1) \end{equation}

Solving Equations \eqref{eq:string1}, \eqref{eq:string2}, and \eqref{eq:string3} for $h(\xi_{K_{p,q}})$ yields
$$h(\xi_{K_{p,q}})= h(\xi_K)+ (1-p)(2g-q-1),$$\noindent as claimed. $\Box$
% In \cite{Contact}, \ons \  showed that for fibered knots, $H_*(\Filt(-Y,K,-g))\cong \Z$, where $-Y$ indicates the manifold $Y$ with opposite orientation, and where $g$ is the genus of the fiber.  They further showed that $H_*(\Filt(-Y,K,i))\cong 0$ for $i<-g$. Let $c_0(K)$ denote a generator for $H_*(\Filt(-Y,K,-g))$, and let $c(K)$ denote its image under the map on homology induced from the natural inclusion map
%$$i : \Filt(-Y,K,-g) \rightarrow \CFa(-Y).$$

\subsection{Determination of fibered cable knots which induce $\std$}

In this section, we use Proposition \ref{prop:SQPfiber} to determine which cable knots induce the standard tight contact structure on $S^3$.  We prove the following:
\begin{thm}\label{thm:tightness}
Let $p>0$, and let $K$ be a non-trivial fibered knot.  Then $\xi_{K_{p,q}}\simeq \std$ if and only if  $\xi_K\simeq \std$ and $q>0$.
\end{thm}

Again, we rely on our previous results concerning the Floer homology of cables.  In this case, we call upon Theorem $1.2$ of \cite{CableII}

\begin{thm}
	\label{thm:tau}

Let $K\subset S^3$ be a knot. Then the following inequality holds for all $n$,
$$  p\cm \tau(K)+ \frac{(pn)(p-1)}{2}\ \  \le\tau(K_{p,pn+1})\le \ \ 
        p\cm \tau(K)+\frac{(pn)(p-1)}{2} +p-1. $$

\noindent In the special case when $K$ satisfies $\tau(K)=g(K)$ we have the equality,
$$  \tau(K_{p,pn+1}) =  p\cm \tau(K)+\frac{(pn)(p-1)}{2}.  $$
\end{thm}
\noindent {\bf Proof of Theorem \ref{thm:tightness}:} 

Recall that Proposition \ref{prop:msum} shows that $K_{p,q}=K_{p,\text{sgn}(q)} * T_{p,q}$.   Torisu's theorem then implies that  $\xi_{K_{p,q}}\simeq \xi_{K_{p,\text{sgn}(q)} }\# \xi_{{T_{p,q}}}$.  In the case that $q<-1$, $\xi_{{T_{p,q}}}$ is overtwisted, as indicated by the fact that $h(\xi_{{T_{p,q}}})=(p-1)(-q-1)$.  Thus $K_{p,q}$ is overtwisted, provided $q<-1$.   Likewise, if $q\ge 1$ then $\xi_{T_{p,q}}\simeq \std$.  Thus $$\xi_{K_{p,q}}\simeq \xi_{K_{p,1} }\# \std \simeq \xi_{K_{p,1}}$$
whenever $q\ge 1$.   

In light of these remarks, it suffices to understand $\xi_{K_{p,\pm1}}$.  

Assume first that $\xi_K\simeq \std$.  Proposition \ref{prop:SQPfiber} indicates that $\tau(K)=g(K)$.  The second part of Theorem \ref{thm:tau} tells us that (for  $n=0$)
$$  \tau(K_{p,1}) =  pg(K)= g(K_{p,1}).$$
\noindent Hence, $\xi_{K_{p,1}}$ is tight, by Proposition \ref{prop:SQPfiber}.  On the other hand,  Theorem \ref{thm:cablehopf} implies
$$  h(\xi_{K_{p,-1}}) =  h(\xi_K)+2g(K)(1-p)=2g(K)(1-p)\ne 0,$$
\noindent  Where the second equality follows from the fact that $\xi_K\simeq \std$, by assumption, and hence has vanishing Hopf invariant. Thus, $\xi_{K_{p,-1}}$ is overtwisted.  This proves Theorem \ref{thm:tightness} in the case that $\xi_K$ is tight. 
 
\begin{figure}[htb]
\begin{center}
\includegraphics{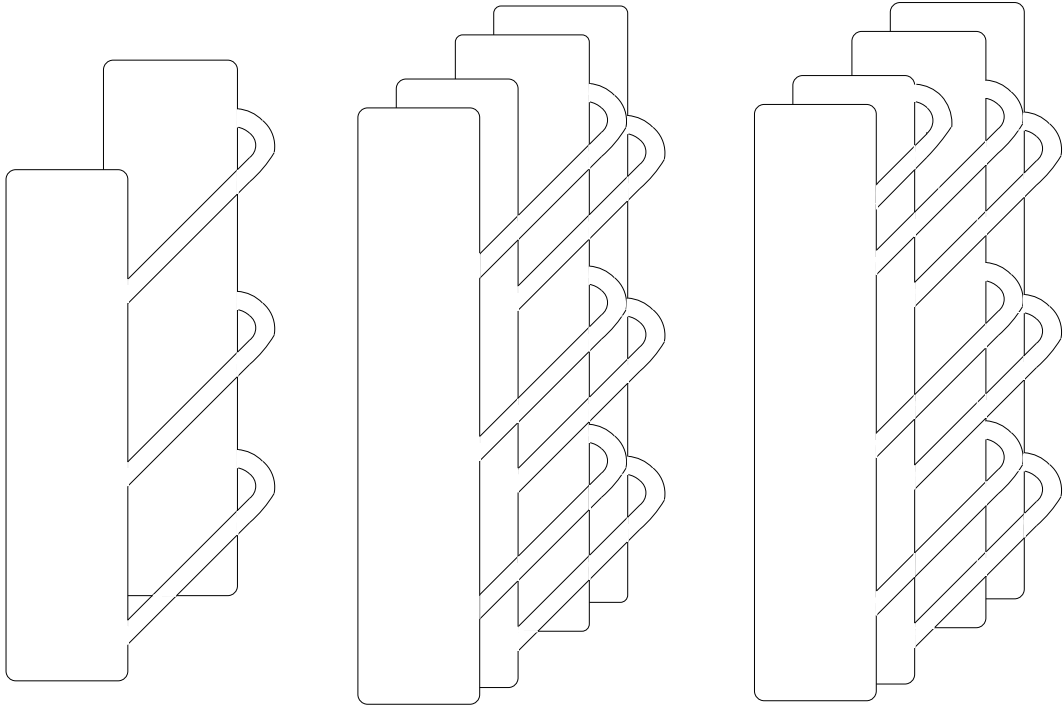}
\caption{Quasipositive Seifert surfaces are surfaces obtained, like the ones pictured here, by stacking parallel disks and attaching positive bands (here, positive means that the bands turn upwards on the right).  This figure depicts the manner in which we can cable a quasipositive Seifert surface.  The left figure is a quasipositive Seifert surface, $F$, for the trefoil.  The middle represents the first surface together with a parallel push-off in a product neighborhood of $F$.  The middle surface is clearly quasipositive and its boundary is the $(2,0)$ cable link of the trefoil.  The third surface is obtained by adding a positive band to the second, and is a quasipositive Seifert surface for the $(2,1)$ cable of the trefoil.}
\label{fig:SQPcable}
\end{center}
\end{figure}

  If $\xi_K$ is not tight, then $\tau(K)< g(K)$ by Proposition \ref{prop:SQPfiber}  (note that $\tau(K)\le g(K)$ by the adjunction inequality for knot Floer homology, Theorem $5.1$ of \cite{Knots}).   Theorem \ref{thm:tau} then shows that $$\tau(K_{p,1})\le p\tau(K)+p-1 < pg(K) = g(K_{p,1}).$$
\noindent Hence $\xi_{K_{p,1}}$ is not tight, by Proposition \ref{prop:SQPfiber}.  To deal with $K_{p,-1}$, note that $K_{p,1}$ can be changed into $K_{p,-1}$ by a sequence of $(p-1)$ crossing changes, each of which change a positive to crossing to a negative.  Corollary $1.5$ of \cite{FourBall} indicates that $\tau(K)$ cannot increase under such crossing changes, and hence
$$\tau(K_{p,-1})\le \tau(K_{p,1})  <  g(K_{p,-1}).$$
\noindent Thus $\xi_{K_{p,-1}}$ is not tight.  This completes the proof. $\Box$

\newpage

\noindent {\bf Proof of Corollary \ref{cor:complex}:} 

The corollary follows immediately from Theorem \ref{thm:tightness} and Proposition \ref{prop:SQPfiber}.  Indeed, Proposition \ref{prop:SQPfiber} characterizes fibered knots which induce $\std$ as precisely those for which the fiber surface is isotopic to a piece of complex curve.  Theorem \ref{thm:tightness} then determines which cable knots induce $\std$. $\Box$
\bigskip

Though we have proved the theorem and corollary, we think it is interesting to note that if $K$ has a quasipositive Seifert surface, then a quasipositive surface for $K_{p,q}$ can be explicitly constructed for $q>0$.  Since quasipositive surfaces are isotopic to pieces of smooth complex curves in the four-ball \cite{Rudolph}, this indicates that positive cabling is, in some sense, a complex operation.   

To construct the aforementioned surfaces, we simply observe that we can ``cable" a quasipositive surface, $F$, by taking $p$ copies of each disk and band used to construct $F$.    See Figure \ref{fig:SQPcable} for an illustration.  More precisely, we consider a neighborhood, $F\times [0,1]$, of $F\subset S^3$ and take the disjoint surfaces $F\times \{x_i\}$ for $p$ distinct points $x_i\in [0,1]$, $i=1,..,p$.  The boundary of the resulting (quasipositive) surface is the $(p,0)$ cable link on $K$.  To obtain the positive cable knots, we simply attach some more positive bands.  Note that this method proves the ``if" direction of Theorem \ref{thm:tightness} without appealing to the second part of Theorem \ref{thm:tau}.  However, we could not obtain a general proof of the  ``only if" direction which avoided the use of Theorem \ref{thm:tau}.

\end{document}